\documentclass[11pt]{amsart}
\usepackage{amsmath}
\usepackage{amsfonts,amssymb,amsthm,amstext,amscd,boxedminipage}  
\usepackage{setspace}

\usepackage{graphicx}
\usepackage {psfrag} 
\usepackage{color}
\usepackage{tikz}

\newlength{\arrowsize}  
\pgfarrowsdeclare{biggertip}{biggertip}{  
  \setlength{\arrowsize}{0.4pt}  
  \addtolength{\arrowsize}{0.5 \pgflinewidth}  
  \pgfarrowsrightextend{0 pt}  
  \pgfarrowsleftextend{-5\arrowsize}  
}{  
  \setlength{\arrowsize}{0.4pt}  
  \addtolength{\arrowsize}{.5\pgflinewidth}  
  \pgfpathmoveto{\pgfpoint{-5\arrowsize}{4\arrowsize}}  
  \pgfpathlineto{\pgfpointorigin}  
  \pgfpathlineto{\pgfpoint{-5\arrowsize}{-4\arrowsize}}  
  \pgfusepathqstroke  
} 
 
 \numberwithin{equation}{section}

\DeclareMathSymbol{\minus} {\mathord}{operators}{"2D} %

\def\sgn#1;#2{\mathbb{S}_{#1,#2}} 
\def\mcg#1;#2{\Gamma_{#1,#2}} 
\def\fg#1;#2{\Pi_{#1,#2}}
\def\tb#1;#2{\mathscr{K}_{\frac{#1}{#2}}}

\def \R {\mathbb{R}}

\theoremstyle{plain}
\newtheorem{theorem}{Theorem}[section]
\newtheorem{lem}[theorem]{Lemma}

\newtheorem*{maintheorem}{Main Theorem}

\theoremstyle{definition}
\newtheorem{df}[theorem]{Definition}

\begin{document}

\title[Ribbon graph moves]{A reduced set of moves on one-vertex ribbon graphs coming from links}
\date{\today}

\author[Abernathy]{Susan Abernathy}
\address{Department of Mathematics, Louisiana State University,
Baton Rouge, LA 70803, USA}
\email{sabern1@tigers.lsu.edu}

\author[Armond]{Cody Armond}
\address{Department of Mathematics, Louisiana State University,
Baton Rouge, LA 70803, USA}
\email{carmond@math.lsu.edu}
 
\author[Cohen]{Moshe Cohen}
\address{Department of Mathematics, Bar-Ilan University,
Ramat Gan 52900, Israel}
\email{cohenm10@macs.biu.ac.il}

\author[Dasbach]{Oliver T. Dasbach}
\address{Department of Mathematics, Louisiana State University,
Baton Rouge, LA 70803, USA}
\email{kasten@math.lsu.edu}

\author[Manuel]{Hannah Manuel}
\address{Department of Mathematics, Louisiana State University,
Baton Rouge, LA 70803, USA}
\email{hmanue3@tigers.lsu.edu}

\author[Penn]{Chris Penn}
\address{Department of Mathematics, Louisiana State University,
Baton Rouge, LA 70803, USA}
\email{coffee@math.lsu.edu}

\author[Russell]{Heather M. Russell}
\address{Department of Mathematics, University of Southern California,
LA, CA 90089-2532, USA}
\email{heathemr@usc.edu}

\author[Stoltzfus]{Neal W. Stoltzfus}
\address{Department of Mathematics, Louisiana State University,
Baton Rouge, LA 70803, USA}
\email{stoltz@math.lsu.edu}

\thanks {The fourth author was supported in part by NSF grants DMS-0806539 and DMS-0456275 (FRG). The third and seventh authors were partially supported by NSF VIGRE grant DMS 0739382.  The last author was also supported by DMS-0456275 (FRG). }

\begin{abstract}
Every link in $\R^3$ can be represented by a one-vertex ribbon graph. We prove a Markov type theorem on this subset of link diagrams. 
\end {abstract}
 
\maketitle

\section{Introduction}
A classical tool in knot theory is the Alexander polynomial. It is naturally related to Seifert surfaces (e.g.  \cite {Lickorish:KnotTheoryBook}). For a link in braid form, the structure of the Seifert surface is particularly simple. The surface can be viewed as a series of twisted bands between a vertically stacked collection of pancake disks. Vogel gives an algorithmic proof of the classical Alexander theorem by showing that every link diagram can be transformed into a closed braid by a sequence of Reidemeister II moves \cite {Vogel:Algorithm}. The Markov moves then describe the equivalence classes on braids given by the link isotopy classes of their closures \cite {Birman:BraidBook}.

A strikingly similar picture is available in the Jones polynomial context. Here the natural topological objects are Turaev surfaces, which are assigned to knot diagrams  \cite {Turaev:SimpleProof, DFKLS:GraphsOnSurfaces}. Ribbon graphs encode the projection of knots on their Turaev surfaces. Similar to the Alexander theorem, every link can be represented by a one-vertex ribbon graph. Algorithmically, this can be accomplished by using only Reidemeister II moves. In this paper we provide a Markov type theorem on one-vertex ribbon graphs.

Given a diagram $D$ for some link in $\R^3$, a state $s$ is a choice of one of the two local smoothings. These choices, called the $A$ and $B$ smoothings of a crossing, are shown in Figure \ref{smoothings}. Hence, there are $2^{\# \textup{ crossings in }D}$ different states $s$ for $D$ with corresponding smoothed diagrams $D_s$. 

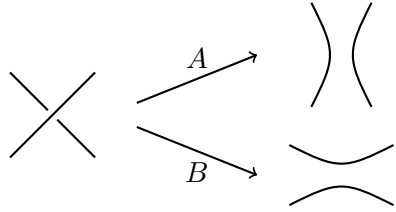
\begin{figure}[ht]
\begin{tikzpicture}[baseline=0cm, scale=0.8]
\draw[style=thick] (135:1)--(315:1);
\draw[color=white, line width=5pt] (45:1) -- (225:1);
\draw[style=thick] (45:1) -- (225:1);
\draw[style=thick, ->]  (1.4,0.2)  -- (2.4,0.6) node[above] {$A$} -- (3.4,1);
\draw[style=thick,->] (1.4,-0.2) -- (2.4,-0.6) node[below] {$B$} -- (3.4,-1);

\draw[xshift=4.8cm, yshift=-1cm, style=thick] (210:1) .. controls (240:0.1) and (300:0.1) .. (330:1);
\draw[xshift=4.8cm, yshift=-1cm, style=thick] (150:1) .. controls (120:0.1) and (60:0.1) .. (30:1);
\draw[xshift=4.8cm, yshift=1cm, style=thick] (300:1) .. controls (330:0.1) and (390:0.1) .. (420:1);
\draw[xshift=4.8cm, yshift=1cm, style=thick] (240:1) .. controls (210:0.1) and (150:0.1) .. (120:1);
\end{tikzpicture}
\caption{$A$ and $B$ smoothings for a link diagram} \label{smoothings}
\end{figure}

Many celebrated knot invariants are computed via a sum over states. In other words, we assign a quantity to each state and then take an appropriate weighted sum of these quantities. Both the Alexander polynomial \cite{Kauffman:OnKnots,CDR:Dimers} and the Jones polynomial \cite{Kauffman:StateModels} can be expressed in this way.  Their categorified counterparts knot-Floer homology (e.g. \cite{MOS:CombinatorialKnotFloer})  and Khovanov homology \cite{Khovanov:homology, Bar-Natan:Khovanov} also have state-sum models. In certain state sum models, the expressions simplify on the subclass of one-vertex ribbon graphs. 

The all-$A$ smoothing $D_A$ corresponds to the state where the $A$ smoothing is chosen for all crossings of $D$. The circles in $D_A$ are in one-to-one correspondence with the vertices of the ribbon graph for the embedding of the diagram $D$ on the Turaev surface.  In general $D_A$ will consist of several circles. The following lemma proves that every link has some diagram for which $D_A$ is a single circle and thus can be represented by a one-vertex ribbon graph \cite{CKS:QuasiTrees}.

\begin{lem}[see also \cite{DFKLS:DDD}] \label{onecircle}
Every link has a diagram for which the all-$A$ smoothing consists of a single circle.
\end{lem}

\begin{proof}
If $D_A$ is a single circle we are done. If $D_A$ consists of more than one circle, then there are two distinct circles in $D_A$ that are adjacent in the sense that they can be connected by an embedded arc in $\R^2-D$.  (Here we understand $\R^2-D$ to be the complement of the 4-valent graph underlying $D$ in the plane.) Perform a Reidemeister II move along that arc as shown in Figure \ref{makeonecirc} to obtain a new diagram $D'$. 

\begin{figure}[ht]
\begin{tikzpicture}[baseline=0cm, scale=1.1]
\draw[style=thick] (110:1)-- (250:1);
\draw[style=thick] (70:1) -- (-70:1);
\draw[style=dashed] (-.25,0) -- (.25,0);
\draw (0,-1.2) node {$D_A$};
\draw[style=thick,->] (1.0,0) -- (2,0);
\draw[xshift=3cm, style=thick] (110:1) .. controls (30:0.4) and (-30:0.4) .. (250:1);
\draw[xshift=3cm, color=white, line width=5pt] (70:1) .. controls (150:0.4) and (210:0.4) .. (-70:1);
\draw[xshift=3cm, style=thick] (70:1) .. controls (150:0.4) and (210:0.4) .. (-70:1);
\draw (3,-1.2) node {$D_A'$};
\end{tikzpicture}\caption{Merging circles along an arc}\label{makeonecirc}\label{RIIa}
\end{figure}
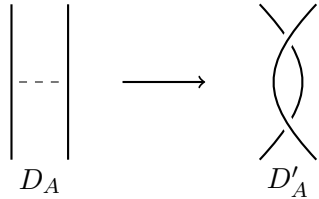

The all-$A$ smoothing of this new diagram, denoted $D'_A$, will have one less circle than $D_A$. In particular $D'_A$ is the result of merging the two adjacent circles from $D_A$ along the embedded arc.  Continue this process of merging along embedded arcs until we arrive at a diagram for which the all-$A$ smoothing is a single circle.
\end{proof}

The set of embedded arcs along which we merge circles is very important in the proof of the main result in this paper. For this reason, we give it the special notation $\mathfrak{S}_D$ and call it a {\it connecting set for $D$}. Section \ref{conn} proves a technical lemma about connecting sets.

Let $\mathcal{D}$ be the set containing all diagrams for all links in $\R^3$; let $\widetilde{\mathcal{D}}$ be the set $\mathcal{D}$ modulo the equivalence relation induced by the three Reidemeister  moves and orientation preserving isotopies in the plane. Then each $[D]\in \widetilde{D}$ represents a unique link $L$ in $\R^3$. 

Now let $\mathcal{D}_1$ be the subset of $\mathcal{D}$ consisting of diagrams for which the all-$A$ smoothing consists of a single circle. It follows from Lemma \ref{onecircle} that for each $L = [D]\in \widetilde{\mathcal{D}}$ there exists some $D'\in \mathcal{D}_1$ with $L = [D] = [D']$. In Section  \ref{MMoves} we present a set of Reidemeister-type moves on the elements of $\mathcal{D}_1$ which we call $M$-moves.

Let $\widetilde{\mathcal{D}_1}$ be the set $\mathcal{D}_1$ under the equivalence relation induced by these new $M$-moves and orientation preserving isotopies of the plane. We can then consider the following two canonical mappings.
$$\phi: \mathcal{D} \rightarrow \widetilde{\mathcal{D}} \textup{ and } \phi_1: \mathcal{D}_1 \rightarrow \widetilde{\mathcal{D}_1}$$ The first one comes from the standard  Reidemeister moves, and the second one comes from the $M$-moves.  In Section \ref{Reid} we prove the following theorem about the relationship between $\widetilde{\mathcal{D}}$ and $\widetilde{\mathcal{D}_1}$ and their associated canonical mappings.
\begin{maintheorem}\label{ReidThm}
Let $D, D'\in \mathcal{D}_1$. Then $\phi(D) = \phi(D')$ if and only if $\phi_1(D) = \phi_1(D')$. 
\end{maintheorem}

This paper was partially inspired by Manturov's work in which the elements of $\mathcal{D}_1$ are further equipped with marked points or base points and studied as bibracket structures \cite{Manturov:BracketCalculus}. 
Manturov proves a theorem similar to our Main Theorem in the language of bi-bracket structures \cite{Manturov:BracketCalculus}. The moves presented here are different than those found in \cite{Manturov:BracketCalculus}. Furthermore, because we present our moves directly on link diagrams in the set  $\mathcal{D}_1$, we eliminate the need for base points.

This project is the outcome of a semester long  NSF-VIGRE research course in knot theory at LSU. We are grateful to VIGRE for providing the organizational resources that made this research possible.

\section{Moves on ``single circle" diagrams}\label{MMoves}
We showed in Lemma \ref{onecircle} that every link in $\R^3$ has a diagram in the set $\mathcal{D}_1$. In this section we present a set of moves on the elements of $\mathcal{D}_1$ which we will call $M$-moves. These moves are well-defined, meaning each move takes a diagram in $\mathcal{D}_1$ to another diagram in $\mathcal{D}_1$.  

For the sake of comparison and because we will refer to them later, we first present the classical Reidemeister moves in Figure \ref{rmoves}. Each Reidemeister  move is subdivided into two types: $a$ and $b$ based on how the number of circles in the all-$A$ smoothing changes under the move.

Call the region where the Reidemeister move is being performed the distinguished region. Call the collection of  points where the boundary of the distinguished region intersects the diagram distinguished points.  Performing the all-$A$ smoothing outside of the distinguished region results in a collection of properly embedded arcs and circles with the arcs connecting the distinguished points.  Where it is important, we indicate these connections by labeling the boundary points with lowercase letters. For Reidemeister I and II we also pay attention to directionality distinguishing between moves and their inverses. For the remainder of our paper we will refer to the Reidemeister moves as $R$-moves.

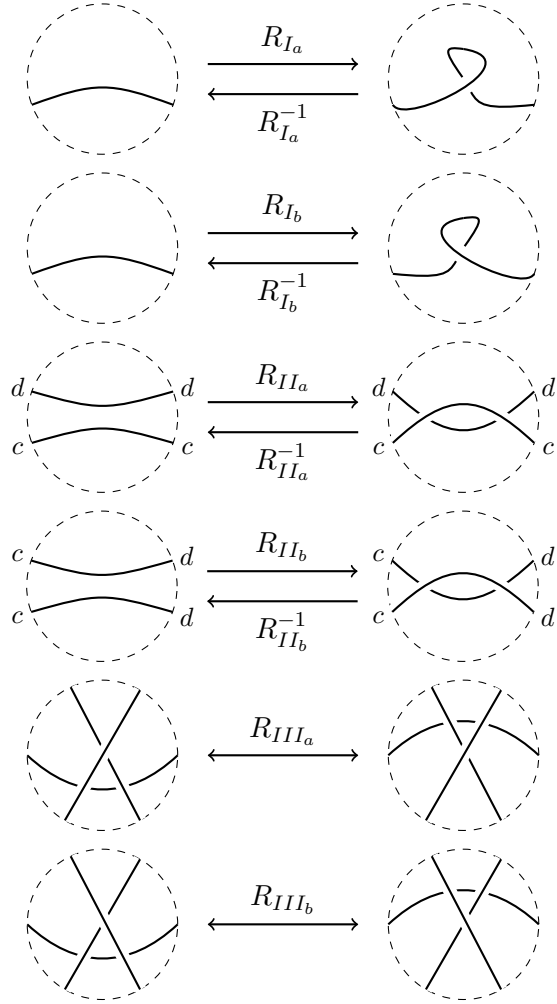
\begin{figure}[ht]
% R Ia
\begin{tikzpicture}[baseline=0cm]
\draw[style=dashed] (0,0) circle (1cm);
\draw[style=thick] (200:1) .. controls (200:0.1) and (340:0.1) .. (340:1);
\draw[style=thick, ->]  (1.4,0.2) -- (2.4,0.2) node[above]{$R_{I_{a}}$} -- (3.4,0.2);
\draw[style=thick,<-] (1.4,-0.2) -- (2.4,-0.2) node[below]{$R_{I_{a}}^{-1}$} -- (3.4,-0.2);
\draw[style=dashed] (4.8,0) circle (1cm);
\draw[xshift=4.8cm, style=thick] (200:1) .. controls (220:1) and (14:1) .. (90:0.4);
\draw[xshift=4.8cm, style=thick] (90:0.4) .. controls (120:0.5) and (125:0.5) .. (120:0.03);
\draw[xshift=4.8cm, style=thick] (300: 0.2) .. controls (300:0.4) and (305:0.5) .. (340:1);
\end{tikzpicture}

\vspace{0.2cm}

%R Ib
\begin{tikzpicture}[baseline=0cm]
\draw[style=dashed] (0,0) circle (1cm);
\draw[style=thick] (200:1) .. controls (200:0.1) and (340:0.1) .. (340:1);
\draw[style=thick, ->]  (1.4,0.2) -- (2.4,0.2) node[above]{$R_{I_{b}}$} -- (3.4,0.2);
\draw[style=thick,<-] (1.4,-0.2) -- (2.4,-0.2) node[below]{$R_{I_{b}}^{-1}$} -- (3.4,-0.2);
\draw[style=dashed] (4.8,0) circle (1cm);
\draw[xshift=4.8cm, style=thick] (340:1) .. controls (320:1) and (166:1) .. (90:0.4);
\draw[xshift=4.8cm, style=thick] (90:0.4) .. controls (60:0.5) and (55:0.5) .. (60:0.03);
\draw[xshift=4.8cm, style=thick] (240: 0.2) .. controls (240:0.4) and (235:0.5) .. (200:1);
\end{tikzpicture}
 
\vspace{0.2cm}

% R IIa
\begin{tikzpicture}[baseline=0cm]
\draw[style=dashed] (0,0) circle (1cm);
\draw[style=thick] (200:1) .. controls (240:0.1) and (300:0.1) .. (340:1);
\draw (200:1.2) node {\small $c$};
\draw (340:1.2) node {\small $c$};
\draw (160:1.2) node {\small $d$};
\draw (20:1.2) node {\small $d$};
\draw[style=thick] (160:1) .. controls (120:0.1) and (60:0.1) .. (20:1);
\draw[style=thick, ->]  (1.4,0.2) -- (2.4,0.2) node[above]{$R_{II_{a}}$} -- (3.4,0.2);
\draw[style=thick,<-] (1.4,-0.2) -- (2.4,-0.2) node[below]{$R_{II_{a}}^{-1}$} -- (3.4,-0.2);
\draw[style=dashed] (4.8,0) circle (1cm);
\draw[xshift=4.8cm, style=thick] (160:1) .. controls (240:0.4) and (300:0.4) .. (20:1);
\draw[xshift=4.8cm, color=white, line width=5pt] (200:1) .. controls (120:0.4) and (60:0.4) .. (340:1);
\draw[xshift=4.8cm, style=thick] (200:1) .. controls (120:0.4) and (60:0.4) .. (340:1);
\draw [xshift=4.8cm] (200:1.2) node {\small $c$};
\draw [xshift=4.8cm]  (340:1.2) node {\small $c$};
\draw  [xshift=4.8cm]  (160:1.2) node {\small $d$};
\draw [xshift=4.8cm]  (20:1.2) node {\small $d$};
\end{tikzpicture}

\vspace{0.2cm}
% R IIb
\begin{tikzpicture}[baseline=0cm]
\draw[style=dashed] (0,0) circle (1cm);
\draw[style=thick] (200:1) .. controls (240:0.1) and (300:0.1) .. (340:1);
\draw[style=thick] (160:1) .. controls (120:0.1) and (60:0.1) .. (20:1);
\draw (200:1.2) node {\small $c$};
\draw (340:1.2) node {\small $d$};
\draw (160:1.2) node {\small $c$};
\draw (20:1.2) node {\small $d$};
\draw[style=thick, ->]  (1.4,0.2) -- (2.4,0.2) node[above]{$R_{II_{b}}$} -- (3.4,0.2);
\draw[style=thick,<-] (1.4,-0.2) -- (2.4,-0.2) node[below]{$R_{II_{b}}^{-1}$} -- (3.4,-0.2);
\draw[style=dashed] (4.8,0) circle (1cm);
\draw[xshift=4.8cm, style=thick] (160:1) .. controls (240:0.4) and (300:0.4) .. (20:1);
\draw[xshift=4.8cm, color=white, line width=5pt] (200:1) .. controls (120:0.4) and (60:0.4) .. (340:1);
\draw[xshift=4.8cm, style=thick] (200:1) .. controls (120:0.4) and (60:0.4) .. (340:1);
\draw [xshift=4.8cm] (200:1.2) node {\small $c$};
\draw [xshift=4.8cm]  (340:1.2) node {\small $d$};
\draw  [xshift=4.8cm]  (160:1.2) node {\small $c$};
\draw [xshift=4.8cm]  (20:1.2) node {\small $d$};
\end{tikzpicture}

\vspace{0.2cm}
% R III a
\begin{tikzpicture}[baseline=0cm]
\draw[style=dashed] (0,0) circle (1cm);
\draw[style=thick] (180:1) .. controls (240:0.7) and (300:0.7) .. (360:1);
\draw[style=thick, color=white, line width=5pt] (115:1) -- (300:1);
\draw[style=thick] (115:1) -- (300:1);
\draw[style=thick, color=white, line width=5pt] (60:1) -- (240:1);
\draw[style=thick] (60:1) -- (240:1);
\draw[style=thick, <->]  (1.4,0) -- (2.4,0) node[above]{$R_{III_{a}}$} -- (3.4,0);
\draw[style=dashed] (4.8,0) circle (1cm);
\draw[xshift=4.8cm, style=thick] (180:1) .. controls (120:0.7) and (60:0.7) .. (360:1);
\draw[xshift=4.8cm, style=thick, color=white, line width=5pt] (115:1) -- (300:1);
\draw[xshift=4.8cm, style=thick] (115:1) -- (300:1);
\draw[xshift=4.8cm, style=thick, color=white, line width=5pt] (60:1) -- (240:1);
\draw[xshift=4.8cm, style=thick] (60:1) -- (240:1);
\end{tikzpicture}

\vspace{0.2cm}
% R III b
\begin{tikzpicture}[baseline=0cm]
\draw[style=dashed] (0,0) circle (1cm);
\draw[style=thick] (180:1) .. controls (240:0.7) and (300:0.7) .. (360:1);
\draw[style=thick, color=white, line width=5pt] (60:1) -- (240:1);
\draw[style=thick] (60:1) -- (240:1);
\draw[style=thick, color=white, line width=5pt] (115:1) -- (300:1);
\draw[style=thick] (115:1) -- (300:1);
\draw[style=thick, <->]  (1.4,0) -- (2.4,0) node[above]{$R_{III_{b}}$} -- (3.4,0);
\draw[style=dashed] (4.8,0) circle (1cm);
\draw[xshift=4.8cm, style=thick] (180:1) .. controls (120:0.7) and (60:0.7) .. (360:1);
\draw[xshift=4.8cm, style=thick, color=white, line width=5pt] (60:1) -- (240:1);
\draw[xshift=4.8cm, style=thick] (60:1) -- (240:1);
\draw[xshift=4.8cm, style=thick, color=white, line width=5pt] (115:1) -- (300:1);
\draw[xshift=4.8cm, style=thick] (115:1) -- (300:1);
\end{tikzpicture}

\caption{The $R$-moves}\label{rmoves}\end{figure}

The six moves shown in Figure \ref{mmoves} are called $M$-moves.  As with the $R$-moves, they are local relations in the sense that we apply an $M$-move within a distinguished region leaving the remainder of the diagram unchanged. The first move, called $M_0$, is the only $M$-move with a connectivity requirement like those found in the $R_{II}$ and $R_{III}$ moves in Figure \ref{rmoves}. Again we indicate connectivity requirements in $M_0$ with lowercase letters.

\begin{figure}[ht]
\begin{center}
\begin{picture}(162,50) 
\put(0,0){\scalebox{.7}{\includegraphics[width=1.2in]{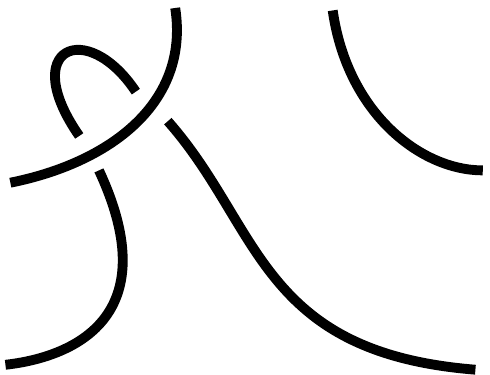}  \hspace{.2in} \raisebox{16pt}{$\stackrel{M_0}{\longleftrightarrow }$}   \hspace{.2in} \includegraphics[width=1.2in]{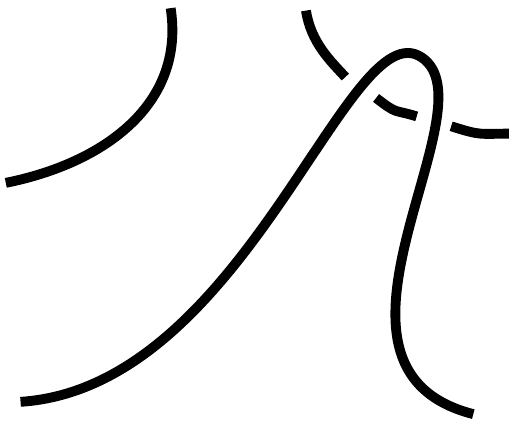}}}
\put(-3,-5){\tiny{$c$}}
\put(58,-5){\tiny{$c$}}
\put(-5,25){\tiny{$a$}}
\put(62,25){\tiny{$a$}}
\put(20,47){\tiny{$b$}}
\put(42,47){\tiny{$b$}}

\put(100,-3){\tiny{$c$}}
\put(158,-5){\tiny{$c$}}
\put(96,25){\tiny{$a$}}
\put(165,29){\tiny{$a$}}
\put(120,50){\tiny{$b$}}
\put(138,50){\tiny{$b$}}

\end{picture}

\vspace{.2in}

\scalebox{.7}{\includegraphics[width=1.2in]{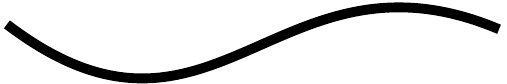} \hspace{.2in} \raisebox{10pt}{$\stackrel{M_{1}}{\longrightarrow}$}\hspace{.2in} \includegraphics[width=1.2in]{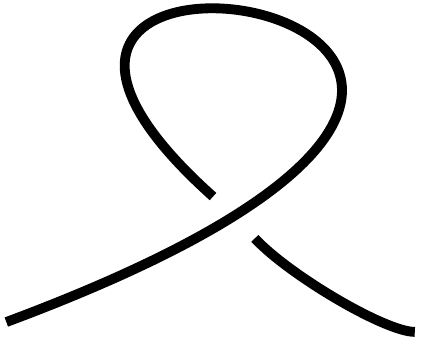}}

\vspace{.2in}

\scalebox{.7}{\includegraphics[width=1.2in]{singlestrand} \hspace{.2in} \raisebox{10pt}{$\stackrel{M_{2}}{\longrightarrow} $}\hspace{.2in} \includegraphics[width=1.2in]{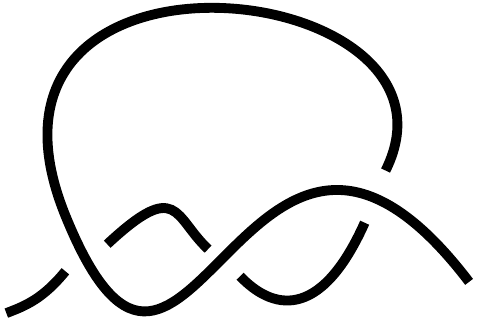}}

\vspace{.2in}

\scalebox{.7}{\includegraphics[width=1.2in]{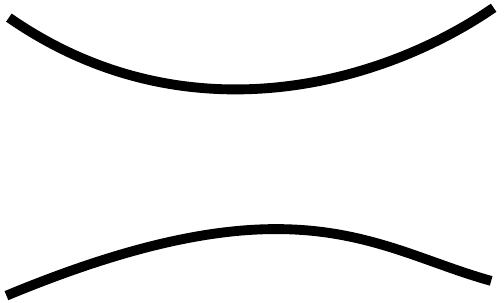} \hspace{.2in} \raisebox{30pt}{$\stackrel{M_3}{\longrightarrow}$}\hspace{.2in} \includegraphics[width=1.2in]{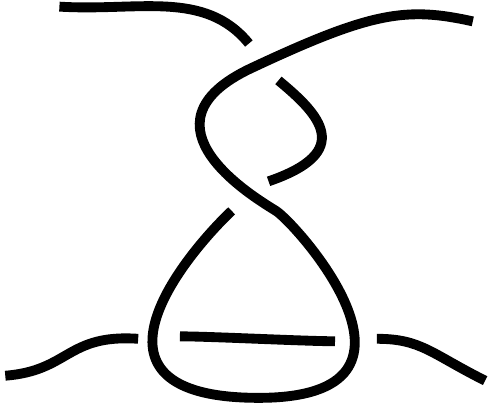}}

\vspace{.2in}

\scalebox{.7}{\includegraphics[width=1.2in]{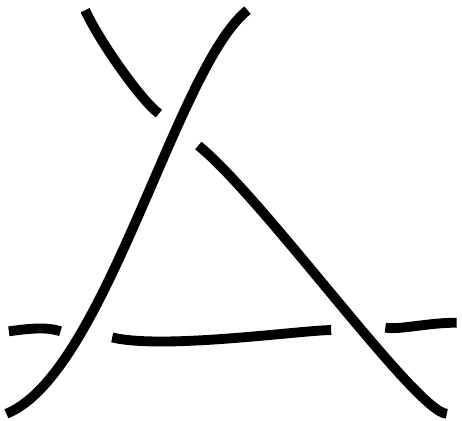} \hspace{.2in} \raisebox{30pt}{$\stackrel{M_4}{\longleftrightarrow} $}\hspace{.2in} \includegraphics[width=1.2in]{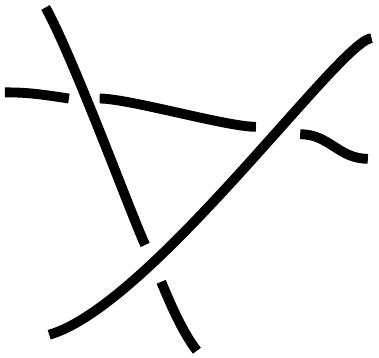}}
\end{center}
\caption{The $M$-moves}\label{mmoves}
\end{figure}

\begin{theorem}
The $M$-moves given in Figure \ref{mmoves} are well-defined on the set $\mathcal{D}_1$.
\end{theorem}

\begin{proof}
To prove this one considers the all-$A$ smoothings within the distinguished region before and after performing each $M$-move. Each smoothing is a disjoint union of embedded arcs connecting the distinguished points pairwise. 

For all but the $M_0$-move well-definedness follows from the fact that, up to isotopy, the all-$A$ smoothings in the distinguished region do not change. 
In the case of the $M_0$-move, the all-$A$ smoothing in the distinguished region changes but the external connectivity requirements guarantee that the new all-$A$ smoothing will still have one circle.  

\end{proof}

\section{Manipulating the connecting set}\label{conn}

In the proof of Lemma \ref{onecircle} we make use of the connecting set $\mathfrak{S}_D$ of a diagram $D$. In this section we prove a result about the flexibility of $\mathfrak{S}_{D}$ that is needed  for the proof of the Main Theorem. We begin by giving a formal definition of the connecting set.
\begin{df}\label{connectingset}
Let $D$ be an arbitrary link diagram. A {\it connecting set} $\mathfrak{S}_D$ for $D$ is a minimal set of arcs embedded in $\R^2-D$ with endpoints on $D$ along which performing $R_{II}$ moves yields an element of $\mathcal{D}_1$. When the diagram is clear from context, we write $\mathfrak{S}$ for the connecting set; in particular, we denote by $D_{\mathfrak{S}}$ the modification of $D$ by prescribed by $\mathfrak{S}$. This diagram is always in $\mathcal{D}_1$.
\end{df}
 
Part of the definition of connecting set is that the arcs it contains must lie in $\R^2-D$, so these arcs cannot intersect the locations of crossings from $D$. As we manipulate elements of the connecting set, we will need to pay special attention to these crossings, so we introduce some notation to simplify this task. 
\begin{df}
A {\it crossing arc} is an embedded arc in $\mathbb{R}^2-D_A$ that marks the location of a crossing from $D$. A set of  crossing arcs, one for each crossing of  $D$ will be called a {\it crossing set} and will be denoted by $\mathfrak{T}_D$. 
\end{df}
In Figure \ref{crossarc} we show a portion of a diagram $D$ and then the same portion decorated with crossing arcs after performing the all-$A$ smoothing. Unlike the connecting set, the crossing set for $D$ is unique. Given $D_A$ and the ccrossing set $\mathfrak{T}_D$, we can reconstruct the diagram $D$ up to isotopy. Moreover, note that $\R^2 -D$ is homeomorphic to $R^2-(D_A\cup \mathfrak{T}_D)$.

\begin{figure}[ht]

\begin{tikzpicture}[baseline=0cm]
\draw[style=thick,<->] (1.4,0)  -- (3.4,0);
\draw[style=thick] (340:1) .. controls (320:1) and (166:1) .. (90:0.4);
\draw[style=thick] (90:0.4) .. controls (60:0.5) and (55:0.5) .. (60:0.03);
\draw[style=thick] (240: 0.2) .. controls (240:0.4) and (235:0.5) .. (200:1);
\draw[xshift=4.8 cm] (0,0.5) circle (0.3 cm);
\draw[xshift=4.8 cm] (200:1) .. controls (200:0.1) and (340:0.1) .. (340:1);
\draw[xshift=4.8cm, line width=1.25pt] (0,-0.1) -- (0,0.2);
\end{tikzpicture}

\caption{Crossing arc}\label{crossarc}
\end{figure}
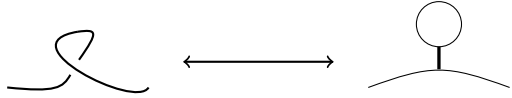

\begin{lem}\label{clearout}
Let $D\in \mathcal{D}$, and let $\mathfrak{S}$ be an arbitrary connecting set for $D$. If $x,y\in D$ are points on distinct circles in $D_A$ connected by an embedded arc $s \in \R^2-D$ that does not intersect any of the elements in $\mathfrak{S}$ then there is another connecting set $\mathfrak{S}'$ for $D$ such that 
\begin{enumerate}
\item{$s\in \mathfrak{S}'$ and} 
\item{the diagrams  $D_{\mathfrak{S}}$ and $D_{\mathfrak{S}'}$ are equivalent via $M_0$-moves.}
\end{enumerate}
\end{lem}
\begin{proof}

We begin by noting that the elements of $\mathfrak{S}$ are actually contained in $\R^2-D_A$. In other words, elements of the connecting set can be viewed as arcs connecting circles in the all-$A$ smoothing. 

Now consider the union $U = D_A \cup \mathfrak{S}$ embedded in $\mathbb{R}^2$ as well as its complement $V = \mathbb{R}^2-U$. The set $V$ will consist of a collection of components in the plane each of which has boundary comprised of some combination of arcs which will alternate between elements of  $\mathfrak{S}$ and portions of circles in $D_A$. Each bounded component of $V$ is contractible since $U$ is connected.

Since $s$ does not intersect any elements of $\mathfrak{S}$ it lies inside of a single component $P$ of $V$. Furthermore, $P-s$ will contain two components. At least one of these two components is bounded and contractible. Call it $P'$. 

The boundary of $P'$ will consist of arcs in $\mathfrak{S}$, portions of circles in $D_A$, and the arc $s$. We will think of this boundary as being a polygon with an even number of sides - every other side being a portion of some circle in $D_A$. As we traverse the boundary of $P'$ some arcs in $\mathfrak{S}_D$ will appear twice. We call these {\it internal arcs}. The others (ones that separate $P'$ from other components of $V-s$) will be called {\it external arcs}. We give $P'$ an orientation. This induces an orientation on the boundary of $P'$ and thus on the external arcs. 
Throughout the remainder of the paper, elements of the connecting set $\mathfrak{S}_D$ will be given by dotted lines and elements of the crossing set $\mathfrak{T}_D$ will be given given by thick solid lines.

Let $m_e$ be the number of edges in the boundary of $P'$ and $m_c$ be the number of crossing arcs inside of $P'$. Let $m=m_e+m_c$. We now give an algorithm which modifies  external arcs and reduces $m$ at each step. Technically we change $P'$ each time but will take liberties with notation and always call it the same thing. 

As we modify the external arcs, we are changing the elements of the connecting set thus producing a sequence of connecting sets $\mathfrak{S} = \mathfrak{S}_0, \mathfrak{S}_1, \ldots, \mathfrak{S}_n$ and an associated sequence of diagrams $D_{\mathfrak{S}}=D_{\mathfrak{S}_0}, \ldots, D_{\mathfrak{S}_n}.$  On the level of the diagrams each of the operations we define comes from the application of an $M_0$ move. 

Eventually these operations will reduce $P'$ to a square for which $m_e=2$ and $m_c=0$. In other words at the final step $P'$ contains no crossing arcs and is bounded by two segments of circles from $D_A$, one arc in $\mathfrak{S}_n$, and the special arc $s$. The arc in $\mathfrak{S}_n$ and the arc $s$ are isotopic. Therefore by performing this isotopy we may take $\mathfrak{S}_n = \mathfrak{S}'$, and the lemma will be proven. 

Beginning with step 0, at each step choose the first external arc $s'$ from $\mathfrak{S}_i$ that appears in the oriented boundary of $P'$ after the arc $s$. Next examine the circle on which the head of $s'$ is incident. Traveling from the head of $s'$ along the circle inside of $P'$ we encounter one of three things:
\begin{enumerate}
\item{a crossing arc from $\mathfrak{T}_D$ in $P'$,}
\item{an connecting arc from $\mathfrak{S}_i$ that is part of the boundary of $P'$,}
\item{ or the arc $s$.}
 \end{enumerate}
 
 Before we examine these cases we mention a fourth possibility we omitted. We could encounter a crossing arc or connecting arc with an endpoint on the circle but lying in a region adjacent to $P'$. We can always use an $M_0$ move to ``jump over" the endpoints of these arcs, so we can ignore them. Furthermore there are only finitely many of these points, so it will not prevent our algorithm from terminating. Figure \ref{jump} shows a picture of this situation.
 
 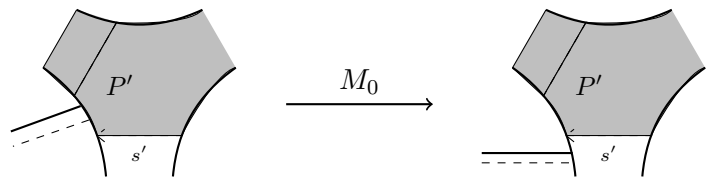
\begin{figure}[ht]
 
 \begin{tikzpicture}[baseline=0cm, scale=1.3]
 \draw[fill=lightgray]  (110:0.68) -- (190:0.68)-- (170:1) -- (130:1);
 \draw[fill=lightgray]  (50:1) ..controls(50:0.78) and (110:0.67) .. (110:0.68) -- (190:0.68) .. controls (200:0.6) and (220:0.6) .. (230:0.68) -- (310:0.68)..controls (310:0.66) and (10:0.8) .. (10:1);
\draw[style=thick] (130:1) .. controls (110:0.6) and (70:0.6) .. (50:1);
\draw[style=thick] (250:1) .. controls (230:0.6) and (190:0.6) .. (170:1);
\draw[style=thick] (10:1) .. controls (350:0.6) and (310:0.6) .. (290:1);
\draw[dashed] (110:0.68) -- (190:0.68);
\draw[dashed, biggertip-] (230:0.68) -- (310:0.68);
\draw( 270:0.70) node{\tiny $s'$};
\draw(180:0.2) node {\small $P'$};

 \draw [style=thick](200:0.64)--(200:1.4);
\draw [dashed](215:0.63)--(206:1.44);
% \draw [style=thick] (200:0.64)--(-1.3,-0.7);
% \draw [dashed] (215:0.63)--(-1.3,-0.8);
\draw[style=thick,->] (1.5,-.2) -- (2.25,-.2) node[above] {$M_{0}$} -- (3,-.2);

 \draw[xshift=4.8cm, fill=lightgray]  (110:0.68) -- (190:0.68)-- (170:1) -- (130:1);
 \draw[xshift=4.8cm, fill=lightgray]  (50:1) ..controls(50:0.78) and (110:0.67) .. (110:0.68) -- (190:0.68) .. controls (200:0.6) and (220:0.6) .. (230:0.68) -- (310:0.68)..controls (310:0.66) and (10:0.8) .. (10:1);
\draw[xshift=4.8cm, style=thick] (130:1) .. controls (110:0.6) and (70:0.6) .. (50:1);
\draw[xshift=4.8cm,style=thick] (250:1) .. controls (230:0.6) and (190:0.6) .. (170:1);
\draw[xshift=4.8cm,style=thick] (10:1) .. controls (350:0.6) and (310:0.6) .. (290:1);
\draw[xshift=4.8cm,dashed] (110:0.68) -- (190:0.68);
\draw[xshift=4.8cm,dashed, biggertip-] (230:0.68) -- (310:0.68);
\draw[xshift=4.8cm] ( 270:0.70) node{\tiny $s'$};
\draw[xshift=4.8cm] (180:0.2) node {\small $P'$};

\draw [xshift=4.8cm, style=thick](-1.3, -0.7)--(-0.37, -0.7);
\draw [xshift=4.8cm, dashed](-1.3,-0.8)--(-0.35 ,-0.8);
\end{tikzpicture}
 
 \caption{Jumping over crossing and connecting arcs not in $P'$}\label{jump}
 \end{figure}
 
\noindent  \underline {Case 1}: We encounter a crossing arc from $\mathfrak{T}_D$ in $P'$.
 
 In this situation there are four possibilities, 
 \begin{enumerate}
 \item[a.]{The crossing arc is parallel to $s'$ with no crossing arcs, connecting arcs, or elements of $D_A$ between them.}
 \item[b.]{The crossing arc and $s'$ are parallel with some crossing arcs, connecting arcs, or elements of $D_A$ inside.}
 \item[c.]{The crossing arc is not parallel to $s'$ and joins two circles that lie in different connected components of $U-\{s'\}$}
 \item[d.]{The crossing arc is not parallel to $s'$ and joins two circles that lie in the same connected component of $U-\{s'\}$.}
 \end{enumerate} 
 
 In Case 1a, we can exchange the positions of the arc and the crossing without changing the diagram $D_{\mathfrak{S}_i}$ as shown in Figure \ref{arccross}.  This reduces $m_c$.
 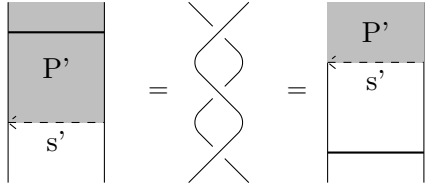
\begin{figure}

 \begin{tikzpicture}[scale=0.8, baseline=0.5cm] 
    \fill [color=lightgray] (-0.8,0.5) rectangle +(1.6, 2);
    \draw (-0.8,-0.5)--(-0.8, 2.5);
    \draw (0.8,-0.5)--(0.8, 2.5); 
    \draw (1.7,1) node{\small =};
    \draw [dashed, biggertip-] (-0.8,0.5)--(0.8,0.5);
    \draw [style=thick] (-.8,2)--(0.8,2);
    \draw (0,0.2) node{s'};
    \draw (0, 1.4) node {P'};
    \end{tikzpicture}
    \begin{tikzpicture}[scale=0.8, baseline=0.5cm, rounded corners=2mm]
    \draw [xshift=3cm](.5, -0.5)-- (0.1,-0.1);
    \draw [xshift=3cm](-.5,-0.5)  -- (0,0)-- (.5,.5)  -- (.1,0.9)--(-.5,1.5)--(0,2)--(0.5, 2.5);
    \draw [xshift=3cm](-.1,.1)--(-.5,.5)--(-.1,0.9);
    \draw [xshift=3cm](0.1,1.1)--(.5,1.5)--(0.1,1.9);
    \draw [xshift=3cm](-0.1, 2.1)--(-0.5, 2.5);
    \draw [xshift=3cm] (1.3,1) node{\small =};
    \end{tikzpicture}   
    \begin{tikzpicture}[scale=0.8, baseline=0.5cm]
 \draw [xshift=6cm] (-0.8,-0.5)--(-0.8, 2.5);
    \draw [xshift=6cm] (0.8,-0.5)--(0.8, 2.5); 
    \draw [xshift=6cm, style=thick] (-0.8, 0)--(0.8,0); 
       \fill [xshift=6cm, color=lightgray] (-0.8, 1.5) rectangle +(1.6, 1);
    \draw [xshift=6cm, biggertip-, dashed] (-0.8, 1.5)--(0.8,1.5);
    \draw [xshift=6cm] (0, 1.2) node{s'}; 
     \draw [xshift=6cm] (0, 2.0) node{P'};
   \end{tikzpicture}  

 \caption{Case 1a: We can switch the position of an arc and a crossing.}\label{arccross}
 \end{figure}
 
 In Case 1b, we can perform an $M_0$ move on the arc $s'$ to move the collection of crossing arcs, connecting arcs, and circles from $D_A$ out of $P'$. This reduces at least one of $m_e$ and $m_c$. 

  In Case 1c, the crossing arc joins two circles in different connected components of $U-\{s'\} = (D_a\cup \mathfrak{S}_i)-\{s'\}$.  Since the connecting set is minimal we know removing $s'$ breaks $U$ into two connected components. The crossing arc in this case has endpoints in both of these components. Thus we can perform an $M_0$ move on the arc $s'$ by leaving the head fixed while moving the tail so that $s'$ and the crossing arc are now parallel. We show this in Figure \ref{Case1c}. This will reduce $m_e$ since at least one internal arc will become an external arc.
 
 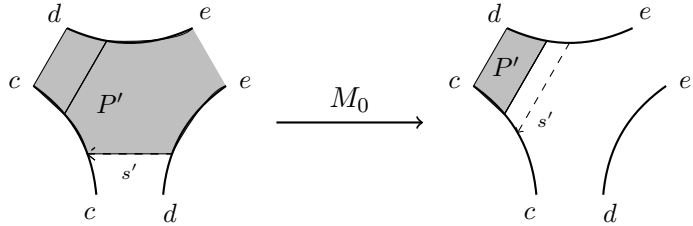
\begin{figure}
 \begin{tikzpicture}[baseline=0cm, scale=1.3]
 \draw[fill=lightgray]  (110:0.68) -- (190:0.68)-- (170:1) -- (130:1);
 \draw[fill=lightgray]  (50:1) ..controls(50:0.78) and (110:0.67) .. (110:0.68) -- (190:0.68) .. controls (200:0.6) and (220:0.6) .. (230:0.68) -- (310:0.68)..controls (310:0.66) and (10:0.8) .. (10:1);
\draw[style=thick] (130:1) .. controls (110:0.6) and (70:0.6) .. (50:1);
\draw (130:1.2) node {\small $d$};
\draw (50:1.2) node {\small $e$};
\draw[style=thick] (250:1) .. controls (230:0.6) and (190:0.6) .. (170:1);
\draw (250:1.2) node {\small $c$};
\draw (170:1.2) node {\small $c$};
\draw[style=thick] (10:1) .. controls (350:0.6) and (310:0.6) .. (290:1);
\draw (10:1.2) node {\small $e$};
\draw (290:1.2) node {\small $d$};
\draw (110:0.68) -- (190:0.68);
\draw[dashed, biggertip-] (230:0.68) -- (310:0.68);
\draw[dashed, biggertip-] (230:0.68) -- (310:0.68);
\draw( 270:0.70) node{\tiny $s'$};
\draw(180:0.2) node {\small $P'$};

\draw[style=thick,->] (1.5,-.2) -- (2.25,-.2) node[above] {$M_{0}$} -- (3,-.2);

\draw[xshift=4.5cm, fill=lightgray]  (110:0.68) -- (190:0.68)-- (170:1) -- (130:1);
\draw[xshift=4.5cm, style=thick] (130:1) .. controls (110:0.6) and (70:0.6) .. (50:1);
\draw[xshift=4.5cm] (130:1.2) node {\small $d$};
\draw[xshift=4.5cm] (50:1.2) node {\small $e$};
\draw[xshift=4.5cm, style=thick] (250:1) .. controls (230:0.6) and (190:0.6) .. (170:1);
\draw [xshift=4.5cm] (250:1.2) node {\small $c$};
\draw [xshift=4.5cm] (170:1.2) node {\small $c$};
\draw[xshift=4.5cm, style=thick] (10:1) .. controls (350:0.6) and (310:0.6) .. (290:1);
\draw [xshift=4.5cm] (10:1.2) node {\small $e$};
\draw [xshift=4.5cm] (290:1.2) node {\small $d$};
\draw [xshift=4.5cm] (110:0.68) -- (190:0.68);

\draw [xshift=4.5cm](110:0.68) -- (190:0.68);
\draw [xshift=4.5cm,dashed,->](90:0.62) -- (210:0.62);
\draw [xshift=4.5cm] (215:0.3) node{\tiny $s'$};

\draw[xshift=4.5cm] (150:0.75) node {\small $P'$};
\end{tikzpicture}

 \caption{Case 1c: Removing $s'$ disconnects the leftmost circle.}\label{Case1c}
 \end{figure}
 
 In Case 1d, the crossing arc joins two circles in the same connected component of $U-\{s'\}$.  Unlike the previous case, the crossing arc has both endpoints on the same component of $U-\{s'\}$. We can still  perform an $M_0$ move, but this time we leave the tail of $s'$ fixed and move the head of $s'$ to a position adjacent to the other end of the crossing arc. We show this in Figure \ref{Case1d}. This reduces $m_c$ by removing the crossing arc from $P'$. It also may reduce $m_e$ by either making some internal edges external or removing some external edges.

  \begin{figure}
 \begin{tikzpicture}[baseline=0cm, scale=1.3]
 \draw[fill=lightgray]  (110:0.68) -- (190:0.68)-- (170:1) -- (130:1);
 \draw[fill=lightgray]  (50:1) ..controls(50:0.78) and (110:0.67) .. (110:0.68) -- (190:0.68) .. controls (200:0.6) and (220:0.6) .. (230:0.68) -- (310:0.68)..controls (310:0.66) and (10:0.8) .. (10:1);
\draw[style=thick] (130:1) .. controls (110:0.6) and (70:0.6) .. (50:1);
\draw (130:1.2) node {\small $e$};
\draw (50:1.2) node {\small $d$};
\draw[style=thick] (250:1) .. controls (230:0.6) and (190:0.6) .. (170:1);
\draw (250:1.2) node {\small $d$};
\draw (170:1.2) node {\small $e$};
\draw[style=thick] (10:1) .. controls (350:0.6) and (310:0.6) .. (290:1);
\draw (10:1.2) node {\small $c$};
\draw (290:1.2) node {\small $c$};
\draw (110:0.68) -- (190:0.68);
\draw[dashed, biggertip-] (230:0.68) -- (310:0.68);
\draw[dashed, biggertip-] (230:0.68) -- (310:0.68);
\draw( 270:0.70) node{\tiny $s'$};
\draw(180:0.2) node {\small $P'$};

\draw[style=thick,->] (1.5,-.2) -- (2.25,-.2) node[above] {$M_{0}$} -- (3,-.2);

\draw[xshift=4.5cm, fill=lightgray]  (70:0.66) -- (350:0.66)-- (10:1) -- (50:1);
\draw[xshift=4.5cm, style=thick] (130:1) .. controls (110:0.6) and (70:0.6) .. (50:1);
\draw[xshift=4.5cm] (130:1.2) node {\small $e$};
\draw[xshift=4.5cm] (50:1.2) node {\small $d$};
\draw[xshift=4.5cm, style=thick] (250:1) .. controls (230:0.6) and (190:0.6) .. (170:1);
\draw [xshift=4.5cm] (250:1.2) node {\small $d$};
\draw [xshift=4.5cm] (170:1.2) node {\small $e$};
\draw[xshift=4.5cm, style=thick] (10:1) .. controls (350:0.6) and (310:0.6) .. (290:1);
\draw [xshift=4.5cm] (10:1.2) node {\small $c$};
\draw [xshift=4.5cm] (290:1.2) node {\small $c$};
\draw [xshift=4.5cm] (110:0.68) -- (190:0.68);

\draw [xshift=4.5cm](110:0.68) -- (190:0.68);
\draw [xshift=4.5cm,dashed,biggertip-](70:0.66) -- (350:0.62);
\draw [xshift=4.5cm] (45:0.3) node{\tiny $s'$};

\draw[xshift=4.5cm] (32:0.75) node {\small $P'$};
\end{tikzpicture}

 \caption{Case 1d: Removing $s'$ disconnects the rightmost circle.}\label{Case1d}
 \end{figure}
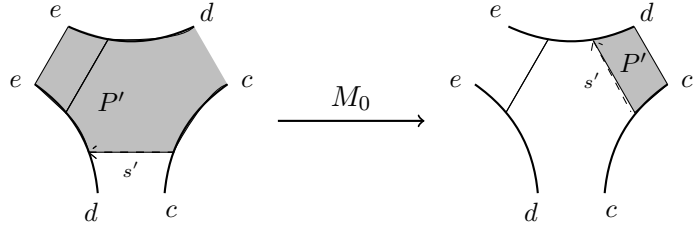
 
 \underline{Case 2}: We encounter another arc from $\mathfrak{S}_i$.
 
 In this situation, there are two possibilities.
 \begin{enumerate}
 \item[a.]{The arc we encounter is internal, or}
 \item[b.]{the arc we encounter is external.}
 \end{enumerate}
 
 In both Cases 2a and 2b, the two circles joined by the arc we encounter are in the same connected component of $U-\{s'\}$. Thus we can perform an $M_0$ move leaving the tail of $s'$ fixed and moving the head as shown in Figure \ref{Case2ab}. In Case 2a, this will change at least one internal edge to an external edge, so $m_e$ will go down. In Case 2b, we remove an external edge from the boundary of $P'$, so again $m_e$ goes down.
 
  \begin{figure}
  
  \begin{tikzpicture}[baseline=0cm, scale=1.3]
 % \draw[fill=lightgray]  (110:0.68) -- (190:0.68)-- (170:1) -- (130:1);
 \draw[fill=lightgray]  (50:1) ..controls(50:0.78) and (110:0.67) .. (110:0.68) -- (190:0.68) .. controls (200:0.6) and (220:0.6) .. (230:0.68) -- (310:0.68)..controls (310:0.66) and (10:0.8) .. (10:1);
\draw[style=thick] (130:1) .. controls (110:0.6) and (70:0.6) .. (50:1);
\draw (130:1.2) node {\small $e$};
\draw (50:1.2) node {\small $d$};
\draw[style=thick] (250:1) .. controls (230:0.6) and (190:0.6) .. (170:1);
\draw (250:1.2) node {\small $d$};
\draw (170:1.2) node {\small $e$};
\draw[style=thick] (10:1) .. controls (350:0.6) and (310:0.6) .. (290:1);
\draw (10:1.2) node {\small $c$};
\draw (290:1.2) node {\small $c$};
\draw [dashed](110:0.68) -- (190:0.68);
\draw[dashed, biggertip-] (230:0.68) -- (310:0.68);
\draw[dashed, biggertip-] (230:0.68) -- (310:0.68);
\draw( 270:0.70) node{\tiny $s'$};
\draw(180:0.2) node {\small $P'$};

\draw[style=thick,->] (1.5,-.2) -- (2.25,-.2) node[above] {$M_{0}$} -- (3,-.2);

\draw[xshift=4.5cm, fill=lightgray]  (70:0.66) -- (350:0.66)-- (10:1) -- (50:1);
\draw[xshift=4.5cm, style=thick] (130:1) .. controls (110:0.6) and (70:0.6) .. (50:1);
\draw[xshift=4.5cm] (130:1.2) node {\small $e$};
\draw[xshift=4.5cm] (50:1.2) node {\small $d$};
\draw[xshift=4.5cm, style=thick] (250:1) .. controls (230:0.6) and (190:0.6) .. (170:1);
\draw [xshift=4.5cm] (250:1.2) node {\small $d$};
\draw [xshift=4.5cm] (170:1.2) node {\small $e$};
\draw[xshift=4.5cm, style=thick] (10:1) .. controls (350:0.6) and (310:0.6) .. (290:1);
\draw [xshift=4.5cm] (10:1.2) node {\small $c$};
\draw [xshift=4.5cm] (290:1.2) node {\small $c$};
\draw [dashed, xshift=4.5cm] (110:0.68) -- (190:0.68);

\draw [xshift=4.5cm](110:0.68) -- (190:0.68);
\draw [xshift=4.5cm,dashed,biggertip-](70:0.66) -- (350:0.62);
\draw [xshift=4.5cm] (45:0.3) node{\tiny $s'$};

\draw[xshift=4.5cm] (32:0.75) node {\small $P'$};
\end{tikzpicture}

 \caption{Cases 2a and 2b: Removing $s'$ will always disconnect the rightmost circle.}\label{Case2ab}
 \end{figure}
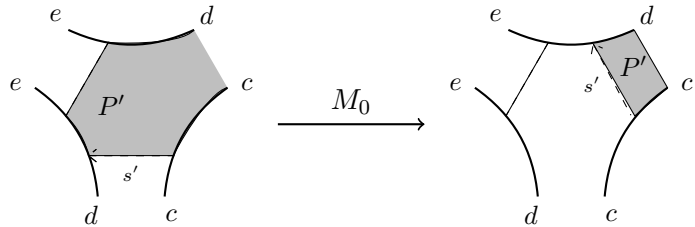

 \underline{Case 3}: The arc we encounter is $s$.
 
 In this situation, we have two possibilities.
 \begin{enumerate}
 \item[a.]{The arcs $s$ and $s'$ are parallel with nothing between them, or}
 \item[b.]{the arcs $s$ and $s'$ are parallel with some undirected arcs, crossing arcs, and circles from $D_A$ between them.}
 \end{enumerate}
 
 These two are analogous to Cases 1a and 1b. In Case 3a since $s$ and $s'$ are isotopic, we can assume that $s$ replaces $s'$ in the connecting set, and we are done. In Case 3b, we can perform an $M_0$ move to remove the undirected arcs, crossing arcs, and circles from $P'$. This will reduce either $m_c$ or $m_e$ depending on what was between $s$ and $s'$.
\end{proof}

\section{A Reidemeister-type theorem on one-vertex ribbon graphs}\label{Reid}
Recall that using the equivalence relations induced by $R$-moves and $M$-moves respectively we get the following two canonical mappings:
$$\phi: \mathcal{D} \rightarrow \widetilde{\mathcal{D}} \textup{ \hspace{.25in} and \hspace{.25in} } \phi_1: \mathcal{D}_1 \rightarrow \widetilde{\mathcal{D}_1}.$$ 
In this section we prove the following theorem which says that we can completely characterize the isotopy classes of links in $\mathbb{R}^3$ by restricting our focus to diagrams in $\mathcal{D}_1$ modulo $M$-moves.
\begin{maintheorem}
Let $D, D'\in \mathcal{D}_1$. Then $\phi(D) = \phi(D')$ if and only if $\phi_1(D) = \phi_1(D')$. 
\end{maintheorem}

One direction of the proof follows immediately from the definition of  the $M$-moves. Indeed the following statement can be verified by examining the pictures of the $M$-moves in Figure \ref{mmoves}.

\begin{lem}\label{easydir}
Given $D, D'\in \mathcal{D}_1$ if $\phi_1(D) = \phi_1(D')$ then $\phi(D) = \phi(D')$.
\end{lem}
\begin{proof}
If $\phi_1(D) = \phi_1(D')$ then we can transform $D$ to $D'$ through a sequence of M-moves. Examining the M-moves in Figure \ref{mmoves} we see that each of them can be expressed as a sequence of $R$-moves. Since we can transform $D$ to $D'$ through a sequence of $R$- moves we conclude that $\phi(D) = \phi(D')$.
\end{proof}

\begin{theorem}
Say $D,D'\in \mathcal{D}_1$ such that $\phi(D) = \phi(D')$. Then $\phi_1(D) = \phi_1(D')$. In other words, two diagrams in $\mathcal{D}_1$ that are related by a sequence of $R$-moves are also related by a sequence of $M$-moves.
\end{theorem}
\begin{proof}
If $\phi(D)=\phi(D')$ then there exists a sequence of diagrams 
$$D = D_0 \overset{R_{0}}{\longrightarrow} D_1 \overset{R_1}{\longrightarrow} \cdots \overset{R_{m-2}}{\longrightarrow}D_{m-1}\overset{R_{m-1}}{\longrightarrow} D_m = D'.$$
In this context each pair of diagrams $D_{i}$ and $D_{i+1}$ differs by a single $R$-move denoted $R_{i}$.

In this proof, we will reinterpret each move $R_{i}$ as a sequence of $M$-moves. We will also build a connecting set $\mathfrak{S}_i$ associated to each diagram $D_i$. Note that since $D,D'\in \mathcal{D}_1$ we see that $\mathfrak{S}_0 = \mathfrak{S}_m = \emptyset$. So we will get a sequence
$$(D_0)_{\emptyset} \rightarrow (D_1)_{ \mathfrak{S}_1} \rightarrow \cdots \rightarrow (D_i)_{ \mathfrak{S}_i} \rightarrow  \cdots \rightarrow (D_{m-1})_{ \mathfrak{S}_{m-1}} \rightarrow (D_m)_{\emptyset}$$
where each arrow is either the identity or a sequence of  $M$-moves.

We consider the $R$-moves listed in Figure \ref{rmoves} and show how each of these can be reinterpreted in terms of a collection of $M$-moves. We begin with diagrams $D_{i}$ and $D_{i+1}$, a connecting set $\mathfrak{S}_i$ for $D_i$, and an $R$-move $R_i$ that takes $D_i$ to $D_{i+1}$. Using a combination of isotopy and repeated applications of Lemma \ref{clearout}, we construct a connecting set $\mathfrak{S}_{i+1}$ to accompany the diagram $D_{i+1}$. Sometimes is is also necessary to have intermediate diagrams and connecting sets.  

We have the following ten cases.

\begin{description}
\item[Case 1] $R_i$ is a $R_{I_a}$-move.
Note that $R_{I_a}$ is just the move $M_{1}$. Using isotopy we can move any elements of $\mathfrak{S}_i$ away from the distinguished region.  Since an application of $R_{I_a}$ does not change the number of circles in the all-$A$ smoothing, $\mathfrak{S}_i$ is also a connecting set for $D_{i+1}$. Set $\mathfrak{S}_{i+1} = \mathfrak{S}_i$ and $M_i = R_i$.

\item[Case 2] $R_i$ is a $R_{I_a}^{-1}$-move.
Note that $R_{I_a}^{-1}$ is just $M_{1}^{-1}$. There are finitely many connecting arcs incident on the loop that $R_i$ removes. By applying a sequence of $M_0$ moves we can move elements of $\mathfrak{S}_i$ out of the distinguished region obtaining a new connecting set for $D_i$ called $\mathfrak{S}_i'$. Set $\mathfrak{S}_{i+1}=\mathfrak{S}_i' $ and $M_i=R_i$.

\item[Case 3]  $R_i$ is a $R_{I_b}$-move.
In this case the number of circles in $(D_i)_A$ is one less than $(D_{i+1})_A$, so the connecting set $\mathfrak{S}_{i+1}$ should have one more arc than the connecting set $\mathfrak{S}_i$. Using isotopy we can move any elements of  $\mathfrak{S}_{i}$ out of the distinguished region. Now let $\mathfrak{S}_{i+1} = \mathfrak{S}_{i} \cup \{s\}$ where the arc $s$ shown in Figure \ref{RIb}. Then $(D_{i+1})_{\mathfrak{S}_{i+1}}$ is related to $(D_i)_{\mathfrak{S}_i'}$ via the move $M_2$.

\begin{figure}[ht]
\begin{tikzpicture}[scale=1.4, baseline=0cm]
\draw[style=thick] (340:1) .. controls (320:1) and (166:1) .. (90:0.4);
\draw[style=thick] (90:0.4) .. controls (60:0.5) and (55:0.5) .. (60:0.03);
\draw[style=thick] (240: 0.2) .. controls (240:0.4) and (235:0.5) .. (200:1);
\draw[dashed] (-.24, 0.05) -- (-.24,-0.3);
\draw (-.35,-0.16) node {\tiny $s$};
\end{tikzpicture}
\caption{$M_2$ is $R_{Ib}$ with an extra connecting arc $s$.} \label{RIb}
\end{figure}
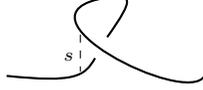

\item[Case 4] $R_i$ is a $R_{I_b}^{-1}$-move.
We begin by invoking Lemma \ref{clearout} to get the element $s$ of the connecting set shown in Figure \ref{RIb}. This forms a new connecting set $\mathfrak{S}_i'$. Now there may be other elements of the connecting set attached to the loop being removed by the move $R_i$. These can be moved off of the loop using a sequence of $M_0$ moves. Call this further modified connecting set $\mathfrak{S}_i''$. Then let $\mathfrak{S}_{i+1} = \mathfrak{S}_i''-\{s\}$. The diagram $(D_{i+1})_{\mathfrak{S}_{i+1}}$ is related to $(D_i)_{\mathfrak{S}_i'}$ via the move ${M_2}^{-1}$.

\item[Case 5] $R_i$ is a $R_{IIa}$-move.
There may be some number of parallel connecting arcs in $\mathfrak{S}_i$ that run through the ``corridor" where the $R_i$ is being performed. Beginning with the leftmost arc we can move each arc out of the corridor via an $M_0$-move that places one of its two endpoints on the lefthand circle involved in the distinguished region. We repeat this until all connecting arcs have been removed from the corridor. This produces a new connecting set $\mathfrak{S}_i'$.
Once the corridor is cleared, we invoke Lemma \ref{clearout} to get a new connecting set containing the arc $s$ shown in Figure \ref{RIIa}. Call this further modified connecting set  $\mathfrak{S}_i''$. Let $\mathfrak{S}_{i+1} = \mathfrak{S}_i''-\{s\}$. The diagram $(D_{i+1})_{\mathfrak{S}_{i+1}}$ is the same as $(D_i)_{\mathfrak{S}_i''}$.

\item[Case 6] $R_i$ is a $R_{IIa}^{-1}$-move.
We use $M_0$-moves to move any connecting arcs out of the distinguished region. This produces a new connecting set $\mathfrak{S}_i'$. Let $\mathfrak{S}_{i+1} = \mathfrak{S}_i'\cup \{s\}$ where $s$ is the connecting arc shown in Figure \ref{RIIa}. Then as in the previous case we have $(D_{i+1})_{\mathfrak{S}_{i+1}}$ is the same as $(D_i)_{\mathfrak{S}_i'}$.

\item[Case 7] $R_i$ is a $R_{IIb}$-move.
As in Case 5, there may be some number of parallel connecting arcs running through the ``corridor" where $R_i$ is to be performed. Using the same process a sequence of $M_0$-moves will remove all connecting arcs from the corridor producing a new connecting set $\mathfrak{S}_i'$.  Now let $\mathfrak{S}_{i+1} = \mathfrak{S}_i' \cup \{s\}$ where $s$ is the arc shown in Figure \ref{RIIb}. Then $(D_i)_{\mathfrak{S}_i'}$ is related to $(D_{i+1})_{\mathfrak{S}_{i+1}}$ via the move $M_3$.

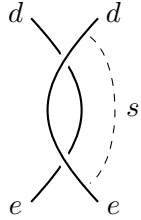
\begin{figure}[ht]

\begin{tikzpicture}[baseline=0cm, scale=1.3]
\draw[style=thick] (110:1) .. controls (30:0.4) and (-30:0.4) .. (250:1);
\draw[color=white, line width=5pt] (70:1) .. controls (150:0.4) and (210:0.4) .. (-70:1);
\draw[style=thick] (70:1) .. controls (150:0.4) and (210:0.4) .. (-70:1);
\draw[dashed] (0.26,0.75) .. controls (0.6,0.5) and (0.6,-0.5) .. (0.26,-0.75);
\draw (-0.5,1) node{$d$};
\draw (0.5,1) node{$d$};
\draw (-0.5,-1) node{$e$};
\draw (0.5,-1) node{$e$};
\draw (0.7,0) node {$s$};
\end{tikzpicture}

\caption{The connecting arc $s$ for $R_{IIb}$}\label{RIIb}
\end{figure}

\item[Case 8] $R_i$ is a $R_{IIb}^{-1}$-move.
We begin by using $M_0$-moves to clear any connecting arcs from the distinguished region. This yields a new connecting set $\mathfrak{S}_i'$. Then we invoke Lemma \ref{clearout} to get a new connecting set $\mathfrak{S}_i''$ containing the arc $s$ shown in Figure \ref{RIIb}. Let $\mathfrak{S}_{i+1} = \mathfrak{S}_i'' - \{s\}$. Then $(D_i)_{\mathfrak{S}_i'}$ is related to $(D_{i+1})_{\mathfrak{S}_{i+1}}$ via the move $M_3^{-1}$.

\item[Case 9] $R_i$ is an $R_{IIIa}$-move.
We first use isotopy and possibly some $M_0$-moves to move any connecting arcs out of the distinguished region except for a possible connecting arc in the triangular region enclosed by the three crossings involved in the move. This yields a new connecting set $\mathfrak{S}_i'$. 

If there is a connecting arc within that triangular region, it will be one of the two shown in Figure \ref{RIIIa}. We invoke Lemma \ref{clearout} to get another arc $s$ which which is also shown in Figure \ref{RIIIa}. Call this new connecting set $\mathfrak{S}_i''$.  Since both the arc in the triangular region and the arc $s$ connect the same circles in $(D_{i})_A$ the connecting arc within the triangular region will not be present in $\mathfrak{S}_i''$. Now let $\mathfrak{S}_i'' = \mathfrak{S}_{i+1}$. Then $(D_i)_{\mathfrak{S}_i''}$ and $(D_{i+1})_{\mathfrak{S}_{i+1}}$ are related by the move $M_4$. 

\begin{figure}[ht]
\begin{tikzpicture}[baseline=0cm]
\draw[style=thick] (180:1) .. controls (240:0.7) and (300:0.7) .. (360:1);
\draw[style=thick, color=white, line width=5pt] (110:1) -- (320:1);
\draw[style=thick] (110:1) -- (320:1);
\draw[style=thick, color=white, line width=5pt] (70:1) -- (220:1);
\draw[style=thick] (70:1) -- (220:1);
\draw[dashed] (-0.3,0)--(0.3,0);
\draw (0,-.1) node {\tiny s};
\draw (0:2) node{\small or};
%\draw [yshift=1.3 cm] (0:2) node {\small before:};

\draw[xshift=4cm, style=thick] (180:1) .. controls (240:0.7) and (300:0.7) .. (360:1);
\draw[xshift=4cm, style=thick, color=white, line width=5pt] (110:1) -- (320:1);
\draw[xshift=4cm, style=thick] (110:1) -- (320:1);
\draw[xshift=4cm, style=thick, color=white, line width=5pt] (70:1) -- (220:1);
\draw[xshift=4cm, style=thick] (70:1) -- (220:1);
\draw[xshift=4cm, dashed] (0.3,0)--(0,-.4);
\draw[xshift=4cm] (0,-0.1) node{\tiny s};

\draw [xshift=4cm, ->] (0:1.7)--(0:3.2);

\draw[xshift=9cm, style=thick] (180:1) .. controls (240:0.7) and (300:0.7) .. (360:1);
\draw[xshift=9cm, style=thick, color=white, line width=5pt] (110:1) -- (320:1);
\draw[xshift=9cm, style=thick] (110:1) -- (320:1);
\draw[xshift=9cm, style=thick, color=white, line width=5pt] (70:1) -- (220:1);
\draw[xshift=9cm, style=thick] (70:1) -- (220:1);
\draw[xshift=9cm, dashed] (-0.25,0.8)--(0,0.8) node[above]{\tiny s} -- (0.25,0.8);
\end{tikzpicture}

\caption{A connecting arc in the triangular region can be replaced by $s$.}\label{RIIIa}
\end{figure}
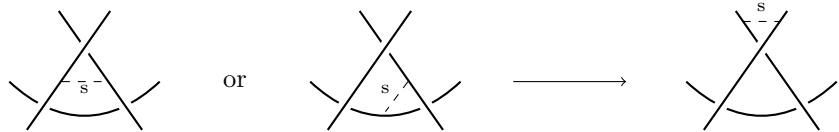

\item[Case 10] $R_i$ is an $R_{IIIb}$-move.
This move can be written as the composition of two $R_{II}$-moves and a $R_{IIIa}$-move. Thus we do not need to consider this case separately.
\end{description}

\end{proof}

\section{Example}

We now give an example of the algorithm we have just described. Figure \ref{exrmove} shows a sequence of $R$-moves performed on the unknot. This sequence begins and ends with elements of $\mathcal{D}_1$, so the Main Theorem tells us that we can construct a sequence of $M$-moves taking the first diagram to the last.  This is shown in Figure \ref{exmmove}.

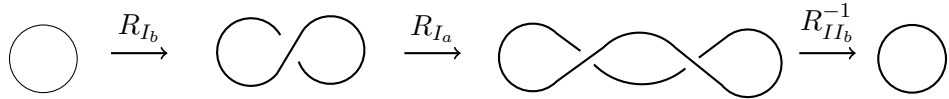
\begin{figure}[ht]
\begin{tikzpicture}[baseline=0cm, scale=1.5]
\draw (-1,-.25) circle (.3cm);
\draw[style=thick,->] (-.4,-.2) -- (-.15,-.2) node[above] {$R_{I_b}$} --
(.1,-.2);
\draw[xshift=.9cm,style=thick,scale=.2]  (1,-.2) arc (30:320:1.5cm) --
(2.0,0) arc (145:-160:1.5cm); % -- cycle;

\draw[xshift=2.2cm,style=thick,->] (0,-.2) -- (0.25,-.2) node[above]
{$R_{I_a}$} -- (0.5,-.2);

 \draw[xshift=4cm,style=thick,scale=.3]  (2.7,-.8)--(3.5,-.2) arc
(130:-140:1cm) -- (2.0,-.5) arc (45:135:1.5cm)--(-1.5,-1.5) arc
(315:45:1cm)--(-.8,-.6);
\draw[xshift=4cm,style=thick,scale=.3]  (-.4,-1.) arc (225:310:2.cm);

\draw[xshift=5.7cm,style=thick,->] (0,-.2) -- (0.25,-.2) node[above]
{$R_{II_b}^{-1}$} -- (0.5,-.2);
\draw[xshift=7.7cm,style=thick] (-1,-.25) circle (.3cm);
\end{tikzpicture}
 \caption {A sequence of $R$-moves taking the crossingless unknot to itself.} \label{exrmove}
 \end{figure}

 \newcommand{\Q}{\mathbb{Q}}
\renewcommand{\d}{\mathcal{D}}
\newcommand{\E}{\mathcal{E}}
\newcommand{\I}{\mathcal{I}}

 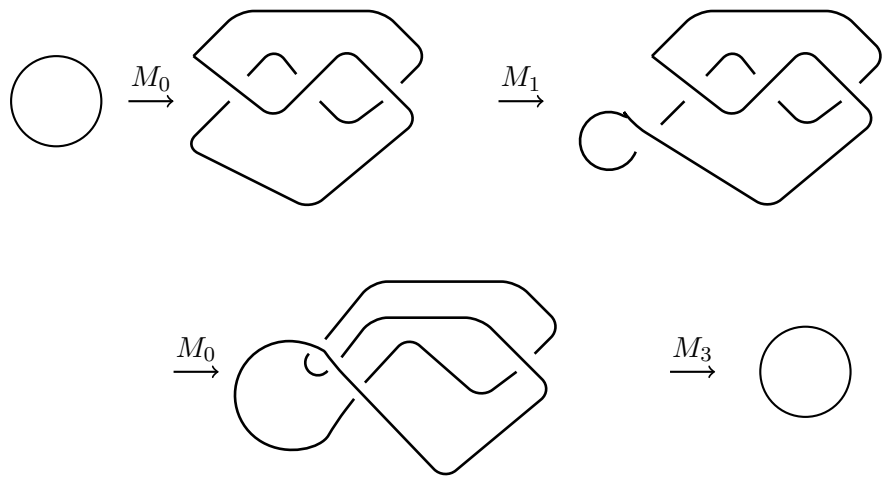
\begin{figure}
\begin{tikzpicture}[baseline=0cm, scale=1.2]
\draw [style=thick](-2.3,-.2) circle (.5cm); 
\draw[xshift=-1.5cm,style=thick,->] (0,-.2) -- (0.25,-.2) node[above] {$M_0$} -- (0.5,-.2);
\begin{scope}[scale=.4,color=black,line width=1pt,rounded corners=2mm] 
 \draw[xshift=1.3cm](-3.25,.75)--(-2,2)--(2,2)--(3.25,.75)--(2.5,0);
 \draw[xshift=1.3cm] (2,-.5)--(1,-1.25)--(.25,-.5);
 \draw[xshift=1.3cm] (-.4,0.25)--(-1,1)--(-1.75,0.2);
\draw[xshift=1.3cm](-2.25,-.5)--(-3.5,-1.75) --(0,-3.5)--(3,-1.)--(1,1)--(-1,-1)--(-3.25,.75);
\end{scope}
\draw[xshift=3.0cm,style=thick,->] (-.4,-.2) -- (-.15,-.2) node[above] {$M_1$} -- (.1,-.2);

\begin{scope}[scale=.4,color=black,line width=1pt,rounded corners=2mm] 
 \draw[xshift=14cm](-3.25,.75)--(-2,2)--(2,2)--(3.25,.75)--(2.5,0);
 \draw[xshift=14cm] (2,-.5)--(1,-1.25)--(.25,-.5);
 \draw[xshift=14cm] (-.4,0.25)--(-1,1)--(-1.75,0.2);
\draw[xshift=14cm](-2.35,-.5)--(-3.,-1.15) (-3.7,-1.9) arc (-22:-325:.8cm)--(0,-3.5)--(3,-1.)--(1,1)--(-1,-1)--(-3.25,.75);
\end{scope}

\draw[xshift=-1cm,yshift=-3cm,style=thick,->] (0,-.2) -- (0.25,-.2) node[above] {$M_0$} -- (0.5,-.2);

\begin{scope}[scale=.4,color=black,line width=1pt,rounded corners=2mm] 
 \draw[xshift=5cm,yshift=-7.5cm](-3.3,.4)--(-2,2)--(2,2)--(3.25,.75)--(2.5,0);
 \draw[xshift=5cm,yshift=-7.5cm]  (2,-.5)-- (1,-1.25)--(-1.,.5)--(-2.2,-.8);  
 \draw[xshift=5cm,yshift=-7.5cm] (-3.75,-.0) arc(135:320:.35cm);
\draw [xshift=5cm,yshift=-7.5cm] (-2.5,-1.25)--(-3.,-1.9) arc (-30:-325:1.5cm)--(0,-3.5)--(3,-1.)--(1,1)--(-2,1)--(-2.85,-.1);
\end{scope}

\draw[xshift=4.5cm,yshift=-3cm,style=thick,->] (0,-.2) -- (0.25,-.2) node[above] {$M_3$} -- (0.5,-.2);
\draw[xshift=7cm,yshift=-3cm,style=thick] (-1,-.2) circle (.5cm); 
\end{tikzpicture}
 \caption {A sequence of $M$-moves taking the crossingless unknot to itself.}\label{exmmove}
 \end{figure}

\section{Application}

As pointed out in the introduction one can assign to each knot diagram a ribbon graph $\Gamma$, i.e. a graph with an embedding on an orientable surface, from which the knot can be recovered (up to reversing the orientation of the ambient space) \cite{DFKLS:GraphsOnSurfaces}.
For an alternating knot with an alternating diagram the graph is plane, and this characterizes alternating knots.
Bollob\'as and Riordan \cite{BR:BRTorientable} introduced a three variable extension to ribbon graphs $C(X,Y_{BR},Z; \Gamma)$ of the Tutte polynomial $T(X,Y_{T}; G)$. 
For a ribbon graph $\Gamma$ whose underlying graph $G_{{\Gamma}}$ is obtained by forgetting the embedding of $\Gamma$ one has
$$C(X, Y_{BR},1;\Gamma)=T(X, Y_{BR}+1;G_{\Gamma}).$$

For a knot diagrams with associated ribbon graph $\Gamma$ the evaluation $$C(-A^4,-A^{-4}-1, (-A^{2}-A^{-2})^{-2};\Gamma)$$ yields the Kauffman bracket of the knot diagram \cite{DFKLS:GraphsOnSurfaces}, which is - up to normalizations and variable change - the Jones polynomial of the knot.   
In the case of one-vertex ribbon graphs $\Gamma$ the Bollob\'as-Riordan-Tutte polynomial simplifies significantly and becomes a two variable polynomial 
$C(-,Y_{BR},Z;\Gamma)$ \cite{BR:BRTorientable}.

Note that in \cite{CKS:QuasiTrees} an expansion of the Bollob\'as-Riordan-Tutte polynomial over certain sub-ribbon graphs called quasi-trees is given (see also \cite{DFKLS:DDD}).
This expansion simplifies significantly for one-vertex ribbon graphs.

Furthermore in \cite{DL:TuraevKhovanovHomology} a Khovanov homology theory for ribbon graphs is worked out that coincides with the regular Khovanov homology if the ribbon graph comes from a knot diagram. It will be interesting to see what simplifications the restriction to the one-vertex case gives.
 
\bibliography{OneVertexRibbon}

\newcommand{\etalchar}[1]{$^{#1}$}
\providecommand{\bysame}{\leavevmode\hbox to3em{\hrulefill}\thinspace}
\providecommand{\MR}{\relax\ifhmode\unskip\space\fi MR }
\providecommand{\MRhref}[2]{%
  \href{http://www.ams.org/mathscinet-getitem?mr=#1}{#2}
}
\providecommand{\href}[2]{#2}
\begin{thebibliography}{DFK{\etalchar{+}}10}

\bibitem[Bir74]{Birman:BraidBook}
Joan~S. Birman, \emph{{Braids, links, and mapping class groups}}, Princeton
  University Press, 1974.

\bibitem[BN02]{Bar-Natan:Khovanov}
Dror Bar-Natan, \emph{{On Khovanov's categorification of the Jones
  polynomial}}, Alg. Geom. Top. \textbf{2} (2002), no.~May, 337--370.

\bibitem[BR01]{BR:BRTorientable}
B\'{e}la Bollob\'{a}s and Oliver~M. Riordan, \emph{{A polynomial invariant of
  graphs on orientable surfaces}}, Proceedings of the London Mathematical
  Society \textbf{83} (2001), no.~3, 513--531.

\bibitem[CDR10]{CDR:Dimers}
Moshe Cohen, Oliver~T. Dasbach, and Heather~M. Russell, \emph{{A twisted dimer
  model for knots}}, arXiv:1010.5228v2 (2010), 16.

\bibitem[CKS11]{CKS:QuasiTrees}
Abhijit Champanerkar, Ilya Kofman, and Neal~W. Stoltzfus, \emph{{Quasi-tree
  expansion for the Bollobas-Riordan-Tutte polynomial}}, Bulletin of the London
  Mathematical Society \textbf{43} (2011), no.~5, 972--984.

\bibitem[DFK{\etalchar{+}}08]{DFKLS:GraphsOnSurfaces}
Oliver~T. Dasbach, David Futer, Efstratia Kalfagianni, Xiao-Song Lin, and
  Neal~W. Stoltzfus, \emph{{The Jones polynomial and graphs on surfaces}}, J.
  Comb. Theory, Ser. B \textbf{98} (2008), no.~2, 384--399.

\bibitem[DFK{\etalchar{+}}10]{DFKLS:DDD}
\bysame, \emph{{Alternating Sum Formulae for the Determinant and Other Link
  Invariants}}, J. Knot Theory Ramifications \textbf{19} (2010), no.~06,
  765--782.

\bibitem[DL11]{DL:TuraevKhovanovHomology}
Oliver~T. Dasbach and Adam~M. Lowrance, \emph{{A Turaev surface approach to
  Khovanov homology}}, arXiv:1107.2344 (2011), 30.

\bibitem[Kau87a]{Kauffman:OnKnots}
Louis~H. Kauffman, \emph{{On Knots}}, Princeton University Press, 1987.

\bibitem[Kau87b]{Kauffman:StateModels}
\bysame, \emph{{State models and the Jones polynomial}}, Topology \textbf{26}
  (1987), no.~3, 395--407.

\bibitem[Kho00]{Khovanov:homology}
Mikhail Khovanov, \emph{{A categorification of the Jones polynomial}}, Duke
  Math. J. \textbf{101} (2000), no.~3, 359--426.

\bibitem[Lic97]{Lickorish:KnotTheoryBook}
W.~B.~Raymond Lickorish, \emph{{An introduction to knot theory}}, Springer,
  1997.

\bibitem[Man02]{Manturov:BracketCalculus}
Vassily~O. Manturov, \emph{{Knots and the Bracket Calculus}}, Acta Appl. Math.
  \textbf{74} (2002), 293--336.

\bibitem[MOS09]{MOS:CombinatorialKnotFloer}
Ciprian Manolescu, Peter Ozsv\'{a}th, and Sucharit Sarkar, \emph{{A
  combinatorial description of knot Floer homology}}, Ann. of Math.
  \textbf{169} (2009), no.~2, 633--660.

\bibitem[Tur87]{Turaev:SimpleProof}
Vladimir~G. Turaev, \emph{{A simple proof of the Murasugi and Kauffman theorems
  on alternating links}}, Enseign. Math.(2) \textbf{33} (1987), no.~3-4,
  203--225.

\bibitem[Vog90]{Vogel:Algorithm}
Pierre Vogel, \emph{{Representation of links by braids: A new algorithm}},
  Comment. Math. Helv. \textbf{65} (1990), no.~1, 104--113.

\end{thebibliography}
\bibliographystyle {amsalpha}
 \end{document}